\documentclass[12pt,a4paper,reqno]{amsart}
\allowdisplaybreaks
\usepackage{amsmath}
\usepackage{amsfonts}
\usepackage{amssymb,amsthm,amsfonts,amsthm,latexsym,enumerate,url,cases}
\usepackage{booktabs,array,longtable}
\numberwithin{equation}{section}
\usepackage{mathrsfs}
\usepackage{hyperref}
\hypersetup{colorlinks=true,citecolor=blue,linkcolor=blue,urlcolor=blue}
\newcommand{\N}{\mathbb N}
\newcommand{\R}{\mathbb R}
\newcommand{\cH}{\mathcal H}
\newcommand{\cC}{\mathcal C}
\newcommand{\cP}{\mathcal P}
\newcommand{\cR}{\mathcal R}
\newcommand{\1}{\mathbf 1}
\newcommand{\bsu}{\boldsymbol u}
\newcommand{\bsy}{\boldsymbol y}
\newcommand{\bsz}{\boldsymbol z}
\newcommand{\eps}{\varepsilon}
\DeclareMathOperator{\vol}{vol}

\theoremstyle{plain}
\newtheorem{theorem}{Theorem}[section]
\newtheorem{lemma}[theorem]{Lemma}
\newtheorem{problem}{Problem}
\newtheorem{corollary}[theorem]{Corollary}
\newtheorem{proposition}[theorem]{Proposition}

\theoremstyle{definition}

\usepackage{etoolbox}
\makeatletter
\patchcmd{\@settitle}{\uppercasenonmath\@title}{}{}{}
\patchcmd{\@setauthors}{\MakeUppercase}{}{}{}
\patchcmd{\section}{\scshape}{}{}{}
\makeatother

\begin{document}

\title[{Products of Two Integers Avoiding Perfect Powers}]{Products of Two Integers Avoiding Perfect Powers}

\author{Quan-Hui Yang}
\address[Quan-Hui Yang]{Ministry of Education Key Laboratory for NSLSCS \\ School of Mathematical Sciences \\ Nanjing Normal University \\ Nanjing 210023 \\ China}
\email{yangquanhui01@163.com}
\author{Lilu Zhao}
\address[Lilu Zhao]
{School of Mathematical Sciences, University of Science and Technology of China, Hefei, Anhui, 230026, P.R. China}
\email{zhaolilu@ustc.edu.cn}


\begin{abstract}
For integers $d\geq 3$, let $F_{2,d}(n)$ be the largest size of a subset
of $[n]$ containing no two distinct elements whose product is a perfect
$d$-th power, and let $f_{2,d}(n)$ denote the analogous quantity when the
two elements need not be distinct. Fleiner, Juh\'asz, K\"ov\'er, Pach,
and S\'andor proved that both complements have order $n^{2/3}$ when
$d=3$, and asked for a leading constant.  They also asked whether, more
generally, $n-F_{k,d}(n)$ and $n-f_{k,d}(n)$ have order $n^{k/d}$ for
$1<k<d$.

We establish asymptotic formula in the case $k=2$ for every fixed $d\geq3$,
\[
 n-F_{2,d}(n)\sim n-f_{2,d}(n)
 \sim C_d\, n^{2/d}(\log n)^{d-3},
\]
where $C_d>0$ is given explicitly by an Euler product and a polytope
volume.  In particular, the extra logarithmic factor gives a negative
answer to the second question for every $d\geq4$.  For $d=3$ we obtain
\[
 C_3=\frac{\pi^2}{4}
 \prod_p\left(1-\frac3{p^2}+\frac2{p^3}\right),
\]
which answers the first question. 
The proof uses an exact decomposition into complementary $d$-free
kernel classes, a squarefree sieve in multiplicative boxes, and a
two-height polytope calculation.
\end{abstract}

 \maketitle

\section{Introduction}

Let $\N=\{1,2,\cdots\}$ denotes the set of positive integers, and write $[n]=\{1,\ldots,n\}$.  For integers $k,d\geq2$, let
$F_{k,d}(n)$ be the largest size of a set $A\subseteq[n]$ for which
\[
 a_1\cdots a_k=x^d,
 \qquad a_1<\cdots<a_k,
\]
has no solution with $a_1,\ldots,a_k\in A$ and $x\in\mathbb{N}$.  The square
case was introduced by Erd\H{o}s, S\'ark\"ozy, and T. S\'os
\cite{ESS}.  Very recently, Fleiner, Juh\'asz, K\"ov\'er, Pach, and S\'andor
\cite{FJKPS} studied cubes and also introduced a closely related
function $f_{k,d}(n)$, in which repetitions among the $a_i$ are allowed,
and the specified trivial solutions are disregarded. Precisely, a solution is called trivial if the multiset $\{a_1,\ldots,a_k\}$ can be partitioned into $d$-element blocks, each block consisting of $d$ copies of the same integer. Gy\H{o}ri \cite{G},  Pach-Vizer \cite{PV} and Tao \cite{Tao} studied the square case (i.e. $d=2$) in this topic. For related results, one may also refer to Pach \cite{Pach2015,Pach2019} and Verstra\"ete \cite{V}.

In this paper, we focus on the case $k=2$. For $d=3$, the following result was proved by Fleiner, Juh\'asz, K\"ov\'er, Pach, and S\'andor \cite{FJKPS}.

\begin{theorem}[Theorem 6 \cite{FJKPS}]\label{thm1}There exist positive constant $c_1$ and $c_2$ such that 
$$c_1n^{2/3}<n-F_{2,3}(n)\le n-f_{2,3}(n)<c_2n^{2/3}.$$
\end{theorem}
As a strengthen of Theorem \ref{thm1}, the authors of \cite{FJKPS} further proposed the following problem. 
\begin{problem}[Problem 35 \cite{FJKPS}]\label{P1}Is it true that there exists a constant $c$ such that 
$$n-f_{2,3}(n)=(c+o(1))n^{2/3}.$$
\end{problem}
Fleiner et al. \cite{FJKPS} also proposed the following problem as an extension of Theorem \ref{thm1}.
\begin{problem}[Problem 34 \cite{FJKPS}]\label{P2}Suppose that $1<k<d$. Is it true that
\begin{align}\label{P2eq}
n^{k/d}\ll n-F_{k,d}(n)\leq n-f_{k,d}(n)\ll n^{k/d}\ ?
\end{align}
\end{problem}

The purpose of this paper is to settle Problem \ref{P1} affirmatively and settle Problem \ref{P2} negatively (when
$k=2$).

To state the constant, put $q=d-1$ and define
\begin{align}
 \Delta_q:=\prod_p\left(1-\frac1p\right)^q
                  \left(1+\frac qp\right).
 \label{eq:Delta}
\end{align}
The product converges absolutely and is positive.  Let
\begin{align}
 \cP_d:=\left\{\bsy\in\R_{\geq0}^{d-1}:
 \sum_{j=1}^{d-1}j y_j=1,
 \quad
 \sum_{j=1}^{d-1}(d-j)y_j=1
 \right\},
\label{eq:polytope}
\end{align}
and set
\begin{align}
 J_d:=d^{3/2}(d-1)\sqrt{\frac{d-2}{12}},
 \qquad
 \alpha_d:=J_d^{-1}d^2\,
 \cH^{d-3}(\cP_d),
\label{eq:alpha}
\end{align}
where $\cH^{d-3}$ is the Hausdorff measure of dimension $d-3$. We remake that $\cH^0$ is the counting measure.  The point
$y_1=\cdots=y_{d-1}=2/(d(d-1))$ lies in $\cP_d$, so
$\alpha_d>0$: the two defining row vectors are linearly independent,
and this all-positive point is a relative-interior point of their
$(d-3)$-dimensional affine fiber.

Throughout, $\mu$ denotes the M\"obius function and let all logarithms be natural unless
another base is displayed. The main result in this paper is the following.

\begin{theorem}[Main theorem]\label{thm:main}
For every fixed integer $d\geq3$, as $n\to\infty$,
\begin{align}
 n-F_{2,d}(n)\sim n-f_{2,d}(n)
 \sim C_d\, n^{2/d}(\log n)^{d-3},
\label{eq:main}
\end{align}
where
\begin{align}
 C_d=\frac{\pi^2}{12}\,\alpha_d\,\Delta_{d-1}.
\label{eq:Cd}
\end{align}
Moreover,
\begin{align}
 F_{2,d}(n)-f_{2,d}(n)=
 \begin{cases}
 1,&d\text{ odd},\\[2mm]
 \displaystyle\sum_{s\leq n^{2/d}}\mu^2(s),&d\text{ even}.
 \end{cases}
\label{eq:exactdifference}
\end{align}
\end{theorem}

The power of $\log n$ in \eqref{eq:main} has a geometric explanation.
After logarithmic coordinates are introduced, the relevant region has
two linear height constraints.  The exponential weight is maximized
where both constraints are equalities, and this maximizing face has
dimension $d-3$.

For $d=3$, the polytope $\cP_3$ is a point and $\alpha_3=3$.
For $d=4$, it is a line segment and $\alpha_4=1$.  We therefore obtain
the following explicit consequences.

\begin{corollary}[Solution of Problem 35 \cite{FJKPS}]\label{cor:d3}
We have
\[
 n-F_{2,3}(n)\sim n-f_{2,3}(n)\sim C_3\,n^{2/3},
\]
where
\[
 C_3=\frac{\pi^2}{4}
 \prod_p\left(1-\frac3{p^2}+\frac2{p^3}\right)
 =0.7075209\ldots.
\]
\end{corollary}

\begin{corollary}[A negative answer to Problem 34 \cite{FJKPS}]\label{cor:d4}
For $d=4$,
\[
 n-F_{2,4}(n)\sim n-f_{2,4}(n)
 \sim C_4\,\sqrt n\log n,
\]
where
\[
 C_4=\frac{\pi^2}{12}
 \prod_p\left(1-\frac6{p^2}+\frac8{p^3}-\frac3{p^4}\right)
 =0.0944883\ldots.
\]
More generally, \eqref{P2eq} is false for $k=2$ and every
$d\geq4$.
\end{corollary}

\section{Complementary kernels and exact formulas}

For $a\in\mathbb{N}$, define its $d$-free kernel by
\[
 \kappa_d(a):=\prod_p p^{v_p(a)\bmod d},
\]
where $v_p(a)\bmod d$ means the residue of $v_p(a)$ modulo $d$, belonging to $\{0,1,\ldots,d-1\}$.
Thus $a=\kappa_d(a)t^d$ uniquely, where $\kappa_d(a)$ is $d$-free.
If $r$ is $d$-free, define its complementary kernel $r^*$ by
\begin{align}
 v_p(r^*)=
 \begin{cases}
 0,&v_p(r)=0,\\
 d-v_p(r),&1\leq v_p(r)\leq d-1.
 \end{cases}
\label{eq:star}
\end{align}
Clearly $(r^*)^*=r$.

\begin{lemma}\label{lem:kernelcriterion}
For positive integers $a$ and $b$, the product $ab$ is a perfect
$d$-th power if and only if
\[
 \kappa_d(b)=\kappa_d(a)^*.
\]
\end{lemma}

\begin{proof}
For every prime $p$, the condition that $v_p(a)+v_p(b)$ be divisible
by $d$ says precisely that the two residues modulo $d$ are either both
zero, or are nonzero and sum to $d$.  This is exactly complementarity
in the sense of \eqref{eq:star}, and the condition for all primes is
equivalent to the assertion.
\end{proof}

For a $d$-free integer $r$, let
\begin{align}
 \cC_r(n):=\{rt^d:t\in\mathbb{Z}^{+},\ rt^d\leq n\},
 \qquad
 c_r(n):=|\cC_r(n)|
 =\left\lfloor\left(\frac nr\right)^{1/d}\right\rfloor.\label{eq:classes}
\end{align}
If $r\neq r^*$, Lemma \ref{lem:kernelcriterion} shows that the
forbidden-pair graph induced by
$\cC_r(n)\cup\cC_{r^*}(n)$ is complete bipartite.  Consequently, an
extremal set takes the whole larger class and omits the smaller one;
if the two classes have the same size, either side may be chosen.
Its contribution to either complement is
\begin{align*}
 \min\{c_r(n),c_{r^*}(n)\}
 =\left\lfloor
 \left(\frac{n}{\max(r,r^*)}\right)^{1/d}
 \right\rfloor.
\end{align*}
If $r=r^*$, the distinct members of $\cC_r(n)$ form a clique in the
forbidden-pair graph for $F_{2,d}$, while every member also has a loop
in the repetition-allowing problem defining $f_{2,d}$.  Thus an
$f_{2,d}$-set takes none of this class, whereas an $F_{2,d}$-set may
take one member.  Different orbits do not interact, so all these local
choices can be made independently.

Let
\begin{align}
 K_d(X):=\#\{r:r\text{ is $d$-free},\ r\leq X,\ r^*\leq X\}
\label{eq:Kddef}
\end{align}
and let
\begin{align}
 E_d(X):=\#\{r:r\text{ is $d$-free},\ r=r^*,\ r\leq X\}.
\label{eq:Eddef}
\end{align}

\begin{proposition}[Exact formulas]\label{prop:exact}
For $d\geq3$,
\begin{align}
 n-f_{2,d}(n)
 =\frac12\sum_{t\leq n^{1/d}}
 \left\{
 K_d\left(\frac{n}{t^d}\right)
 +E_d\left(\frac{n}{t^d}\right)
 \right\},
\label{eq:exactf}
\end{align}
and
\begin{align}
 n-F_{2,d}(n)
 =\frac12\sum_{t\leq n^{1/d}}
 \left\{
 K_d\left(\frac{n}{t^d}\right)
 +E_d\left(\frac{n}{t^d}\right)
 \right\}-E_d(n).
\label{eq:exactF}
\end{align}
Furthermore,
\begin{align}
 E_d(X)=
 \begin{cases}
 1,&d\text{ odd and }X\geq1,\\[1mm]
 \displaystyle\sum_{s\leq X^{2/d}}\mu^2(s),&d\text{ even}.
\end{cases}
\tag{2.8}\label{eq:Edevaluation}
\end{align}
As usual, $1$ is regarded as squarefree.
\end{proposition}

\begin{proof}Throughout the proof of this lemma, the letter $r$ denotes $d$-free number. The arguments below \eqref{eq:classes} give
$$f_{2,d}(n)=\sum_{r< r^\ast}|c_r(n)|.$$
Then 
$$n-f_{2,d}(n)=\sum_{r\geq r^\ast}|c_r(n)|=\sum_{t\leq n^{1/d}}\sum_{r^\ast\leq r\leq \frac{n}{t^d}}1.$$
Since the number of $d$-free numbers $r$ satisfying $r^\ast\leq r\leq X$ is equal to the number of $d$-free numbers $r$ satisfying $r\leq r^\ast\leq X$, we conclude that
the number of $d$-free numbers $r$ satisfying $r^\ast\leq r\leq X$ is equal to $\frac{1}{2}(K_d(X)+E_d(X))$. This proves \eqref{eq:exactf}.

For
$F_{2,d}$, one (and only) element can be retained from each set $ \cC_r(n)$ where $r$ satisfies $r=r^*$. This proves \eqref{eq:exactF}. 

Clearly $r=r^\ast>1$ implies that $d$ is even. Thus $E_d(X)=1$ if $d$ is odd.
If $d$ is even, $E_d(X)$ is exact the number of squrefree positive integers no more that $X^{2/d}$. This proves
\eqref{eq:Edevaluation}.
\end{proof}

Every $d$-free kernel has a unique factorization
\begin{align}
 r(\bsu)=\prod_{j=1}^{d-1}u_j^j,
\label{eq:factorization}
\end{align}
where $u_1,\ldots,u_{d-1}$ are pairwise coprime and squarefree, with
the convention that $1$ is squarefree.  Under this parametrization,
\begin{align}
 r(\bsu)^*=\prod_{j=1}^{d-1}u_j^{d-j}.
\label{eq:oppositefactorization}
\end{align}
Thus $K_d(X)$ is a lattice-point count with two monomial height
conditions.  The next two sections determine its asymptotics.

\section{Squarefree tuples in multiplicative boxes}

For $q\geq2$, let $\mathcal S_q$ denote the set of tuples
$(u_1,\ldots,u_q)\in\N^q$ for which the coordinates are squarefree and
pairwise coprime.  Equivalently, $u_1\cdots u_q$, namely the product of $u_1,\ldots,u_q$, is squarefree. The following is essentially an application of a squarefree sieve in multiplicative boxes. 

\begin{lemma}\label{lem:boxsieve}
Fix $q\geq2$ and $\lambda>1$.  Uniformly as
$\min_i A_i\to\infty$,
\begin{align*}
 \#\left(
 \mathcal S_q\cap\prod_{i=1}^q(A_i,\lambda A_i]
 \right)
 =\left(\Delta_q+o_{q,\lambda}(1)\right)
 \prod_{i=1}^q\#\bigl((A_i,\lambda A_i]\cap\N\bigr),
\end{align*}
where $\Delta_q$ is defined in \eqref{eq:Delta}.
\end{lemma}

\begin{proof} Since $\sum_{t^2\mid u}\mu(t)=1$ or $0$ according to $u$ is squarefree or not, we have
\begin{align}\label{bound0}\#\left(
 \mathcal S_q\cap\prod_{i=1}^q(A_i,\lambda A_i]
 \right)=\sum_{t}\mu(t)\sum_{\substack{t^2\mid u_1\cdots u_q \\ u_i\in (A_i,\lambda A_i]}}1.\end{align}
The letter $t$ will denote squarefree numbers. We introduce the set 
$$B(t_1,\ldots,t_q)=\{(u_1,\ldots,u_q)\in \prod_{i=1}^q\bigl((A_i,\lambda A_i]\cap\N\bigr):\ t_i\mid u_i ~(1\leq i\leq q)\},$$
 and we have
$\# B(t_1,\ldots,t_q) \leq \frac{1}{t_1\cdots t_q}\prod_{i=1}^q(\lambda A_i)$. 
If $t^2\mid u_1\cdots u_q$, then there exists $(t_1,\ldots,t_q)\in \N^q$ such that $t_1\cdots t_q=t^2$ and $(u_1,\ldots,u_q)\in B(t_1,\ldots,t_q)$. It is well-known that $\#\{(t_1,\ldots,t_q)\in \N^q:\ t_1\cdots t_q=t^2\}\ll t^{o(1)}$. Therefore, 
\begin{align}\label{bound1}\sum_{\substack{t^2\mid u_1\cdots u_q \\ u_i\in (A_i,\lambda A_i]}}1\ll_{q,\lambda} \frac{1}{t^{2-o(1)}}A_1\cdots A_q.\end{align}

Let $A_*:=\min_iA_i$ and let $D=(A_\ast)^{1/3}$. We introduce the set 
$$C(v_1,\ldots,v_q)=\{(u_1,\ldots,u_q)\in \prod_{i=1}^q\bigl((A_i,\lambda A_i]\cap\N\bigr):\ u_i\equiv v_i\bmod {t^2} ~(1\leq i\leq q)\},$$
For $t\leq D$, we have 
\begin{align*}\#\{u\in (A_i,\lambda A_i): u\equiv v\bmod {t^2}\}=&\, \frac{1}{t^2}\#\bigl((A_i,\lambda A_i]\cap\N\bigr)+O(1)
\\ =&\,\frac{1}{t^2}\#\bigl((A_i,\lambda A_i]\cap\N\bigr)\big(1+O_{\lambda}(D^{-1})\big),\end{align*}
and therefore,
$\# C(v_1,\ldots,v_q) =\frac{1}{t^{2q}} \prod_{i=1}^q\#\bigl((A_i,\lambda A_i]\cap\N\bigr)\big(1+O_{q,\lambda}(D^{-1})\big)$. 
Let 
$$\rho(t^2)=\#\{(v_1,\ldots,v_q)\in \N^q: 1\leq v_i\leq t^2(1\leq i\leq q)\ \textrm{ and }\ t^2\mid v_1\cdots v_q\}.$$
Then we have 
$$\sum_{t\leq D}\mu(t)\sum_{\substack{t^2\mid u_1\cdots u_q \\ u_i\in (A_i,\lambda A_i]}}1=\prod_{i=1}^q\#\bigl((A_i,\lambda A_i]\cap\N\bigr)\Big(\sum_{t\leq D}\frac{\mu(t)\rho(t^2)}{t^{2q}}+O_{q,\lambda}(D^{-1})\sum_{t\leq D}\frac{\rho(t^2)}{t^{2q}}\Big).$$
Note that $\rho(t^2)$ is multiplicative in the sense that $\rho(p_1^2\cdots p_j^2)=\rho(p_1^2)\cdots \rho(p_j^2)$ for distinct primes $p_1,\ldots,p_j$. 
For a prime $p$, we have 
$$\rho(p^2)=p^{2q}-(p^2-p)^q-q(p-1)(p^2-p)^{q-1}\ll_{q}p^{2q-2}.$$
Thus, $\rho(t^2)=t^{2q-2+o_{q}(1)}$. Now we obtain
\begin{align}\label{bound2}\sum_{t\leq D}\mu(t)\sum_{\substack{t^2\mid u_1\cdots u_q \\ u_i\in (A_i,\lambda A_i]}}1=\prod_{i=1}^q\#\bigl((A_i,\lambda A_i]\cap\N\bigr)
\times\Big(\sum_{t=1}^\infty\frac{\mu(t)\rho(t^2)}{t^{2q}}+O_{q,\lambda}(D^{-1+o(1)})\Big).\end{align}

We observe that $\sum_{t=1}^\infty\frac{\mu(t)\rho(t^2)}{t^{2q}}=\prod_{p}(1-\frac{\rho(p^2)}{p^{2q}})$ and 
\begin{align*} 1-\frac{\rho(p^2)}{p^{2q}}=
 \left(1-\frac1p\right)^q
 +q\left(\frac1p-\frac1{p^2}\right)
       \left(1-\frac1p\right)^{q-1}
 =\left(1-\frac1p\right)^q\left(1+\frac qp\right).
\end{align*}
In particular $\sum_{t=1}^\infty\frac{\mu(t)\rho(t^2)}{t^{2q}}=\Delta_q>0$.

When $t>D$, we use the bound \eqref{bound1}. Now combining \eqref{bound0}, \eqref{bound1} and \eqref{bound2}, we complete the proof of the lemma.
\end{proof}

We will also need a uniform upper bound, which will be used in the dyadic argument. 
\begin{lemma}\label{lem:dyadicH}
Let $d\geq3$ and $X\geq2$. Let $H\geq 0$ be a nonnegative integer. Write $L_0=\log_2X$ and $M=\lfloor2L_0/d\rfloor$.
\begin{align}
 \sum_{\substack{a_1,\ldots,a_{d-1}\geq0\\
 \sum j a_j\leq L_0\\
 \sum(d-j)a_j\leq L_0
 \\ M-(a_1+\cdots +a_{d-1})\geq H}}
 2^{a_1+\cdots+a_{d-1}}
 \ll_d X^{2/d}(1+\log X)^{d-3}
       \sum_{h\geq H}2^{-h}(h+1).
\label{eq:dyadicsum}
\end{align}
\end{lemma}

\begin{proof}
Put $s=a_1+\cdots+a_{d-1}$, and $W=\sum ja_j$.
The two inequalities (involving $L_0$) imply
\begin{align}
 d s-L_0\leq W\leq L_0,
 \qquad s\leq\frac{2L_0}{d}.
\label{eq:dyadicconstraints}
\end{align}
For fixed integers $s$ and $W$, the number of exponent vectors is
$O_d((s+1)^{d-3})$: after $d-3$ coordinates have been chosen, the
remaining two are determined by the two linear equations.  Write $h=M-s$, the interval for $W$ in
\eqref{eq:dyadicconstraints} has at most $d(h+1)$ integer points.  Hence the
left side of \eqref{eq:dyadicsum} is at most
\[
 \ll_d\sum_{s=0}^{M-H}2^s(M-s+1)(s+1)^{d-3}
 \ll_d2^M(M+1)^{d-3}
       \sum_{h\geq H}2^{-h}(h+1),
\]
which proves the assertion.
\end{proof}

In the special case $H=0$, we obtain the following.
\begin{lemma}\label{lem:dyadic}
For fixed $d\geq3$ and $X\geq 1$,
\[
 \sum_{\substack{a_1,\ldots,a_{d-1}\geq0\\
 \sum j a_j\leq\log_2X\\
 \sum(d-j)a_j\leq\log_2X}}
 2^{a_1+\cdots+a_{d-1}}
 \ll_d X^{2/d}(1+\log X)^{d-3}.
\]
\end{lemma}
\begin{proof}When $X\geq 2$, the desired bound follows from Lemma \ref{lem:dyadicH} by taking $H=0$. When $1\leq X\leq 2$, the desired bound holds obviously.\end{proof}

\section{The two-height region}

Let $q=d-1$ and define
\begin{align}
 \cR_d(X):=\left\{\boldsymbol x\in[1,\infty)^{d-1}:
 \prod_{j=1}^{d-1}x_j^j\leq X,
 \quad
 \prod_{j=1}^{d-1}x_j^{d-j}\leq X
 \right\}.\label{eq:defineRd}
\end{align}

\begin{lemma}[Volume of the two-height region]\label{lem:volume}
For every fixed $d\geq3$,
\begin{align}
 \vol(\cR_d(X))
 \sim\alpha_d\, X^{2/d}(\log X)^{d-3},
\label{eq:volumeasymp}
\end{align}
where $\alpha_d$ is defined in \eqref{eq:alpha}.
\end{lemma}

\begin{proof}
Write $L=\log X$ and $x_j=e^{z_j}$.  Then
\[
 \vol(\cR_d(X))
 =\int_{\substack{z_j\geq0\\
 \sum jz_j\leq L\\
 \sum(d-j)z_j\leq L}}
 e^{z_1+\cdots+z_{d-1}}\,d\bsz.
\]
After $z_j=Ly_j$, this becomes
\begin{align}\vol(\cR_d(X))=
 L^{d-1}\mathcal{I} \ \ \textrm{ with }\ \ \mathcal{I}=
 \int_{\substack{y_j\geq0\\A\bsy\leq1\\B\bsy\leq1}}
 e^{L(y_1+\cdots+y_{d-1})}\,d\bsy,
\label{eq:logintegral}
\end{align}
where
\[
 A=(1,2,\ldots,d-1),
 \qquad B=(d-1,d-2,\ldots,1).
\]
Since $A+B=d\1$, we have
\begin{align}
 y_1+\cdots+y_{d-1}=\frac{A\bsy+B\bsy}{d}.
\label{eq:sumheights}
\end{align}

We shall apply the coarea formula to the linear map
$T\bsy=(A\bsy,B\bsy)$.  The Jacobian matrix is
$\begin{pmatrix}
A \\ B
 \end{pmatrix}$, and every two by two submatrix is of rank $2$. Its Jacobian is
\begin{align}
 \sqrt{\det
 \begin{pmatrix}
 A\cdot A&A\cdot B\\B\cdot A&B\cdot B
 \end{pmatrix}}
 =d^{3/2}(d-1)\sqrt{\frac{d-2}{12}}=J_d.
\label{eq:Jcalculation}
\end{align}
Let
\[
 h(s,t):=\cH^{d-3}
 \{\bsy\in\R_{\geq0}^{d-1}:A\bsy=s,\ B\bsy=t\}.
\]
Note that $[0,1]^2$ is contained in the image of $T$. Indeed, $[0,1]^2\subseteq \{T\mathbf{y}:\ \mathbf{y}\in [0,1]^n\}$. Therefore, by the box spline theorem, the function $h(s,t)$ is continuous on $[0,1]^2$ for $d\geq 4$. In fact, for $d=3$, the map $T$ is invertible and $h\equiv1$ on $[0,1]^2$. Moreover, $h$ is understood to be zero if $s<0$ or $t<0$.

 Applying the coarea formula and using
\eqref{eq:sumheights}-\eqref{eq:Jcalculation}, we obtain
\begin{align*}\mathcal{I}=
 \frac1{J_d}\int_{s\leq1,\,t\leq1}
 e^{L(s+t)/d}h(s,t)\,ds\,dt.
\end{align*}
 On
the square make the substitution
\[
 u=L(1-s),\qquad v=L(1-t),
\]
and obtain 
\begin{align*}
 \mathcal{I}=&\,  \frac1{J_d}\cdot\frac{e^{2L/d}}{L^2}
 \int_0^{ L}\!\int_0^{ L}
 e^{-(u+v)/d}
 h\left(1-\frac{u}{L},1-\frac{v}{L}\right)\,du\,dv
 \\ =&\,\frac1{J_d}\cdot \frac{X^{2/d}}{L^2}
 \int_0^{ +\infty}\!\int_0^{ \infty}
 e^{-(u+v)/d}
 h\left(1-\frac{u}{L},1-\frac{v}{L}\right)\,du\,dv,
\end{align*}
where recall that $h$ is understood to be zero if $s<0$ or $t<0$. 

Now by \eqref{eq:logintegral}, we arrive at
$$\frac{\vol(\cR_d(X))}{ X^{2/d}L^{d-3}}=\frac{1}{J_d}\times \int_0^{ +\infty}\!\int_0^{ \infty}
 e^{-(u+v)/d}
 h\left(1-\frac{u}{L},1-\frac{v}{L}\right)\,du\,dv.$$
Dominated
convergence therefore shows that the double integration is asymptotic to
\[
h(1,1)
 \int_0^\infty\!\int_0^\infty e^{-(u+v)/d}\,du\,dv
 =d^2\,h(1,1)=d^2\, \cH^{d-3}(\cP_d),
\]
where $\cP_d$ is defined in \eqref{eq:polytope}.
Therefore, we conclude that
\[
 \vol(\cR_d(X))
 \sim
 L^{d-3}X^{2/d}\frac{d^2\,\cH^{d-3}(\cP_d)}{J_d},
\]
which is \eqref{eq:volumeasymp}. This completes the proof of the lemma.
\end{proof}

\begin{lemma}\label{lem:coordinateboundary}
Let $V_{d,\eps}(X)$ be the volume of the points in $\cR_d(X)$ for
which $\min_jx_j<X^\eps$, and let $N_{d,\eps}(X)$ be the number of
integer points with the same property.  Then
\begin{align}
 \lim_{\eps\rightarrow 0^{+}}\limsup_{X\to\infty}
 \frac{V_{d,\eps}(X)}{X^{2/d}(\log X)^{d-3}}
 =
 \lim_{\eps\rightarrow 0^{+}}\limsup_{X\to\infty}
 \frac{N_{d,\eps}(X)}
 {X^{2/d}(\log X)^{d-3}}=0.
 \label{eq:coordinateboundary}
\end{align}
\end{lemma}

\begin{proof}
Put $L_0=\log_2X$ and place $x_j$ in the dyadic interval
$[2^{a_j},2^{a_j+1})$.  As in Lemma \ref{lem:dyadicH}, write
\[
 s=\sum_{j=1}^{d-1}a_j, \qquad W=\sum_{j=1}^{d-1} ja_j, 
 \qquad M=\left\lfloor\frac{2L_0}{d}\right\rfloor.
\]
Every box meeting $\cR_d(X)$ has $s\leq M$.  Each such box
has volume exactly $2^s$ and contains exactly $2^s$ integer points.
Thus the same estimates below apply simultaneously to volume and to
lattice points.  If such a box meets the set
$\min_jx_j<X^\eps$, then $a_j<\eps L_0$ for at least one coordinate
$j$. As in the proof of Lemma \ref{lem:dyadicH}, for fixed integers $s$ and $W$, the number of exponent vectors is
now $O_d((\eps L_0+1)(s+1)^{d-4})$: after $d-3$ coordinates, including the coordinate $j$ with $a_j<\eps L_0$,  have been chosen, the
remaining two are determined by the two linear equations. Thus, we obtain the upper bound $\ll_d \eps X^{2/d}(1+\log X)^{d-3}$. This proves \eqref{eq:coordinateboundary} for $d\geq 4$.

For $d=3$, without loss of generality, we may assume that $a_1\leq \eps L_0$. we deduce that 
$M-(a_1+a_2)=M-\frac{a_1}{2}-\frac{1}{2}(a_1+2a_2)\geq M-\frac{\eps}{2}L_0-\frac{1}{2}L_0\geq \frac{1-3\eps}{6}L_0-1:=H$. Now by Lemma \ref{lem:dyadicH}, we obtain
$$V_{d,\eps}(X),\, N_{d,\eps}(X)\ll _d X^{2/d}(1+\log X)^{d-3}\sum_{h\geq H}\frac{h+1}{2^h}.$$
Note that $\sum_{h\geq H}\frac{h+1}{2^h}\rightarrow 0~(X\rightarrow +\infty)$. This proves \eqref{eq:coordinateboundary} for $d=3$. The proof of this lemma is complete.

\end{proof}

\begin{proposition}\label{prop:Kd}
For fixed $d\geq3$,
\begin{align}
 K_d(X)\sim
 \alpha_d\Delta_{d-1}X^{2/d}(\log X)^{d-3}.
\label{eq:Kdasymp}
\end{align}
Uniformly for $X\geq1$,
\begin{align}
 K_d(X)\ll_dX^{2/d}(1+\log X)^{d-3}.
\label{eq:Kdupper}
\end{align}
\end{proposition}

\begin{proof}
By \eqref{eq:Kddef}, \eqref{eq:factorization},
\eqref{eq:oppositefactorization} and \eqref{eq:defineRd}, $K_d(X)$ is the number of points of
$\mathcal S_{d-1}$ in $\cR_d(X)$.  The upper bound
\eqref{eq:Kdupper} follows at once by placing each coordinate in a
dyadic interval, discarding the arithmetic restrictions, and applying
Lemma \ref{lem:dyadic}.

We next transfer the density from Lemma \ref{lem:boxsieve} to the
region.  We give a lattice-point sandwich because the region is not hyperrectanglar‌.  Put
\[
 q=d-1,\qquad
 M_d=\sum_{j=1}^{d-1}j=\frac{d(d-1)}2,\qquad
 Q_d(X)=X^{2/d}(\log X)^{d-3}.
\]
Fix $\eps>0$ and $0<\eta\leq1$, put $\lambda=1+\eta$, and keep these parameters
fixed while $X\to\infty$.  Define the central region
\[
 D_\eps(X):=\{\boldsymbol x\in\cR_d(X):
                    x_j\geq X^\eps\text{ for every }j\}
\]
and partition the positive orthant into the disjoint multiplicative
boxes
\begin{align*}
 B_{\boldsymbol m}:=
 \prod_{j=1}^{q}(\lambda^{m_j},\lambda^{m_j+1}],
 \qquad \boldsymbol m\in\mathbb Z^q.
%\label{eq:multiplicativeboxes}
\end{align*}
Let $\mathcal G$ be the family of boxes meeting $D_\eps(X)$, let
$\mathcal I\subseteq\mathcal G$ consist of those boxes contained in
$D_\eps(X)$, and put
\[
 U:=\bigcup_{B\in\mathcal I}B,
 \qquad E:=\bigcup_{B\in\mathcal G\setminus\mathcal I}B.
\]
Every box in $\mathcal G$ has all lower endpoints at least
$X^\eps/\lambda$.  If $A_*(B)$ denotes the smallest lower endpoint of
$B$, then
\begin{align*}
 \#(B\cap\mathbb Z^q)
 =\bigl(1+O_{q,\eta}(A_*(B)^{-1})\bigr)\vol(B).
%\label{eq:ordinaryboxcount}
\end{align*}
Indeed, $U\cup E\subseteq\cR_d(X\lambda^{M_d})$, whose volume is
$O_{d,\eta}(Q_d(X))$ by Lemma \ref{lem:volume}.
Together with the quantitative uniformity in Lemma
\ref{lem:boxsieve}, this shows, after summing over the disjoint boxes,
that
\begin{align}
 \#(\mathcal S_q\cap U)
 =\Delta_q\vol(U)+o_{\eps,\eta}(Q_d(X))
\label{eq:interiorboxdensity}
\end{align}
and
\begin{align}
 \#(E\cap\mathbb Z^q)
 =\vol(E)+o_{\eps,\eta}(Q_d(X)).
\label{eq:boundarylatticevolume}
\end{align}

We now bound the boundary union $E$.  If a box in $\mathcal G$ is not
contained in $\cR_d(X)$, comparison of any two of its points
coordinate by coordinate shows that the whole box lies in
\[
 \cR_d(X\lambda^{M_d})
 \setminus\cR_d(X\lambda^{-M_d}).
\]
If instead it crosses one of the artificial boundaries
$x_j=X^\eps$, then every point of that box has the corresponding
coordinate less than $\lambda X^\eps$.  Consequently,
\begin{align}
 E\subseteq
 \left(\cR_d(X\lambda^{M_d})
             \setminus\cR_d(X\lambda^{-M_d})\right)
 \cup
 \left\{\boldsymbol x\in\cR_d(X\lambda^{M_d}):
                  \min_jx_j<\lambda X^\eps\right\}.
\label{eq:boundaryboxcontainment}
\end{align}
Lemma \ref{lem:volume} gives
\begin{align}
 \vol(\cR_d(X\lambda^{M_d}))
 -\vol(\cR_d(X\lambda^{-M_d}))
 =&\,
 \left(\alpha_d
   (\lambda^{2M_d/d}-\lambda^{-2M_d/d})+o(1)\right)Q_d(X)\notag
 \\= &\, O_d(\eta)Q_d(X)+o_{d,\eta}(Q_d(X)).
\label{eq:monomialshell}
\end{align}
For fixed $\eps,\eta$ and sufficiently large $X$, if
$Y=X\lambda^{M_d}$, then $\lambda X^\eps<Y^{2\eps}$ and
$Q_d(Y)\sim\lambda^{2M_d/d}Q_d(X)$.  The volume assertion in Lemma
\ref{lem:coordinateboundary}, applied at $Y$, therefore implies
\begin{align}
 \limsup_{X\to\infty}\frac{\vol\big\{\boldsymbol x\in\cR_d(X\lambda^{M_d}):
                  \min_jx_j<\lambda X^\eps\big\}}{Q_d(X)}
 \leq \omega_{d,\eta}(\eps)
\label{eq:boundaryvolumesecond}
\end{align}
for some $\omega_{d,\eta}(\eps)$ satisfying $\omega_{d,\eta}(\eps)\longrightarrow 0
 \quad(\eps\rightarrow 0^{+})$.
Now \eqref{eq:boundaryboxcontainment}, \eqref{eq:monomialshell} and \eqref{eq:boundaryvolumesecond} together give
\begin{align}
 \limsup_{X\to\infty}\frac{\vol(E)}{Q_d(X)}
 \leq O_d(\eta)+\omega_{d,\eta}(\eps), \qquad \omega_{d,\eta}(\eps)\longrightarrow 0
 \quad(\eps\rightarrow 0^{+}).
\label{eq:boundaryvolume}
\end{align}
Equation \eqref{eq:boundarylatticevolume} supplies the corresponding
lattice-point bound.

Since $D_\eps(X)\setminus U\subseteq E$,
\eqref{eq:interiorboxdensity} and \eqref{eq:boundaryvolume} give
\begin{align}
 \#(\mathcal S_q\cap D_\eps(X))
 =\Delta_q\vol(D_\eps(X))
 +\Big(O_d(\eta)+O(\omega_{d,\eta}(\eps))+o_{\eps,\eta}(1)\Big)Q_d(X).
\label{eq:centraldensity}
\end{align}
Finally, the part of $\cR_d(X)$ outside $D_\eps(X)$ is negligible for
both volume and lattice points by Lemma
\ref{lem:coordinateboundary}.  After enlarging
$\omega_{d,\eta}(\eps)$ in \eqref{eq:centraldensity} to absorb these small-coordinate errors,
we have
\[
 \limsup_{X\to\infty}
 \frac{|K_d(X)-\Delta_q\vol(\cR_d(X))|}{Q_d(X)}
 \leq O_d(\eta)+O(\omega_{d,\eta}(\eps)).
\]
Letting first $\eps\rightarrow 0^{+}$ and then $\eta\rightarrow 0^{+}$ proves
\[
 K_d(X)=\Delta_{d-1}\vol(\cR_d(X))+o(Q_d(X)).
\]
Combining this with Lemma \ref{lem:volume} proves
\eqref{eq:Kdasymp}.
\end{proof}

\section{Proof of the main theorem}

Let $T=n^{1/2d}$. We now apply \eqref{eq:Kdasymp} in Proposition \ref{prop:Kd} to deduce that
\begin{align}\label{t1}\sum_{t\leq T}K_d\left(\frac{n}{t^d}\right)
 =&\, \Big(\alpha_d\Delta_{d-1}+o_{d}(1)\Big)
 n^{2/d}(\log n)^{d-3}\sum_{t\leq T}\frac{1}{t^2}+O_d(n^{2/d}(\log n)^{d-4})\notag
 \\ 
 =&\, \Big(\frac{\pi^2}{6}\alpha_d\Delta_{d-1}+o_{d}(1)\Big)
 n^{2/d}(\log n)^{d-3}.\end{align}
For $t>T$, we apply the uniform upper bound \eqref{eq:Kdupper} to obtain
\begin{align}\label{t2}\sum_{t>T}K_d\left(\frac{n}{t^d}\right)
 \ll_d &\, 
 n^{2/d}(\log n)^{d-3}\sum_{t>T}\frac{1}{t^2}\ll_d \frac{1}{T}
 n^{2/d}(\log n)^{d-3}.\end{align}
Put \eqref{t1} and \eqref{t2} together, we obtain
\begin{align*}\sum_{t\leq n^{1/d}}K_d\left(\frac{n}{t^d}\right)
 =&\, \Big(\frac{\pi^2}{6}\alpha_d\Delta_{d-1}+o_{d}(1)\Big)
 n^{2/d}(\log n)^{d-3}.\end{align*}

It remains to deal with $E_d(X)$.  If
$d$ is odd, \eqref{eq:Edevaluation} gives $E_d(X)=1$ and thus
\begin{align*}
 \sum_{t\leq n^{1/d}}E_d(n/t^d)\leq n^{1/d}=o(n^{2/d}(\log n)^{d-3}).
%\label{eq:selfodd}
\end{align*}
If $d$ is even, then $E_d(X)\leq X^{2/d}$, and hence (now $d\geq 4$)
\begin{align*}
 \sum_{t\leq n^{1/d}}E_d(n/t^d)
 \leq n^{2/d}\sum_{t\geq1}\frac1{t^2}
 \ll n^{2/d}=o(n^{2/d}(\log n)^{d-3}).
%\label{eq:selfeven}
\end{align*}
Now the proof of
\eqref{eq:main} is complete.  Formula
\eqref{eq:exactdifference} was proved in Proposition
\ref{prop:exact}.  This completes the proof of Theorem
\ref{thm:main}.

\medskip

We conclude by making the first two cases completely explicit.
\subsection{Cubes}
For $d=3$, the equations defining $\cP_3$ have the unique solution
$(1/3,1/3)$.  Also $J_3=3$, so $\alpha_3=3$.  Since
\[
 \Delta_2
 =\prod_p\left(1-\frac1p\right)^2\left(1+\frac2p\right)
 =\prod_p\left(1-\frac3{p^2}+\frac2{p^3}\right),
\]
formula \eqref{eq:Cd} gives
\[
 C_3=\frac{\pi^2}{12}\,\alpha_3\,\Delta_{2}
 =\frac{\pi^2}{4}\Delta_2.
\]
This proves Corollary
\ref{cor:d3}.

\subsection{Fourth powers}
For $d=4$,
\[
 \cP_4=\{(x,\tfrac12-2x,x):0\leq x\leq\tfrac14\}.
\]
This segment has length $\sqrt6/4$, while $J_4=4\sqrt6$.
Consequently $\alpha_4=1$.  Moreover,
\[
 \Delta_3
 =\prod_p\left(1-\frac1p\right)^3\left(1+\frac3p\right)
 =\prod_p\left(1-\frac6{p^2}+\frac8{p^3}-\frac3{p^4}\right).
\]
Thus 
$$C_4=\frac{\pi^2}{12}\,\alpha_4\,\Delta_{3}=\frac{\pi^2}{12}\Delta_3.$$
Theorem \ref{thm:main} and the value of $C_4$ prove Corollary
\ref{cor:d4}.

\section*{Acknowledgments}
This work is support by the National Key Research and Development Program of China (Grant No. 2021YFA1000701), National Natural Science Foundation of China  (Grant No. 12371005 and 12471088).

\end{document}